\documentclass[10pt]{amsart}

\usepackage{amssymb,amsmath,amsthm}
\usepackage{tikz}
\usepackage{tikz-cd}
\usepackage[arc,all]{xy}
\usepackage{eucal}
\usepackage{color}
\usepackage{comment}
\usepackage[colorlinks=true,linkcolor=blue,urlcolor=black,bookmarks=false,bookmarksopen=false]{hyperref}

\theoremstyle{plain}

\newcounter{ssec}[section]

\newtheorem{cor}[ssec]{Corollary}
\newtheorem{lem}[ssec]{Lemma}
\newtheorem{thm}{Theorem}
\newtheorem{prop}[ssec]{Proposition}
\newtheorem{notation}[ssec]{Notations}

\theoremstyle{definition}

\theoremstyle{remark}

\newcounter{srem}[section]

\newtheorem{rem}[srem]{Remark}

\newcommand{\Z}{\mathbb{Z}}

\newcommand{\lp}{_{(p)}}
\newcommand{\Zp}{\mathbb{Z}_{(p)}}

\newcommand{\opn}{\operatorname}

\numberwithin{equation}{section}

\title[On the $p$-primary subgroups of $H^*(BPU_n)$]{On The $p$-primary subgroups of the cohomology of the classifying spaces of $PU_n$}

\author[Z. Zhang]{Zhilei Zhang}
\address{Department of Mathematics, Nankai University, No.94 Weijin Road, Tianjin 300071, P. R. China}
\email{15829207515@163.com}

\author[L. Zhong]{Linan Zhong$^*$}
\address{Department of Mathematics, Yanbian University, Yanji 133000, P. R. China}
\email{lnzhong@ybu.edu.cn}



\thanks{Supported by National Natural Science Foundation of China (Grant No. \textcolor{blue}{12261091;12001474}).}

\begin{document}

\maketitle

\begin{abstract}
    Let $PU_n$ denote the projective unitary group of rank $n$, and let $BPU_n$ be its classifying space. We extend our
    previous results to a description of $H^s(BPU_n;\Z)\lp$ for $s<2p+9$ by showing that $p$-primary subgroups of $H^s(BPU_n;\Z)$ is $\Z/p$ for $s=2p+5$ and are trivial for $s = 2p+7$ and $s = 2p+8$, where $p$ is an odd 
    prime.
\end{abstract}

\section{Introduction}
Let $U_n$ denote the group of $n\times n$ unitary matrices. The projective
unitary group $PU_n$ is defined as the quotient group of $U_n$ by $S^1$,
with the identification of $S^1$ as the normal subgroup of scalar matrices
of $U_n$. Let $BPU_n$ denote the classifying space of $PU_n$.

The cohomology of $BPU_n$ plays significant roles in the study of the topological period-index problem, such as \cite{antieau2014period}, \cite{antieau2014topological}, \cite{gu2019topological}, and \cite{gu2020topological}.
It's also crucial in the study of anomalies in particle 
physics, such as \cite{cordova2020anomalies}, \cite{garcia2019dai}. Other related works include 
\cite{duan}, which fully determined the integral cohomology
of $PU_n$ and \cite{crowley2021h}, which studied the image of the canonical
map $H^* (BPU_n ;\Z)\rightarrow H^* (BU_n ;\Z)$.

The  cohomology of $BPU_n$ for special $n$ has been
studied in various works, such as Kameko-Yagita \cite{Kameko2008brown}, Kono-Mimura \cite{kono1975cohomology}, Kono-Yagita \cite{Kono_Yagita}, Toda \cite{toda1987cohomology}, and Vavpeti{\v{c}}-Viruel \cite{vavpetivc2005mod}. 

None of the works above studied $H^* (BPU_n ;\Z)$ for arbitrary $n$. However, in \cite{gu2019cohomology}, Gu made a breakthrough
which determined the ring structure of $H^*(BPU_n;\Z)$ in dimensions less than or equal to $10$ for any $n$.

\begin{notation}
Let $H^* (-)$ denote the integral cohomology $H^*(-;\Z)$.  Given an abelian group $A$ and a prime number $p$, let $A\lp$ denote the localization of $A$ at $p$, and $_pA$ denote the $p$-primary subgroup of $A$.  In other words,
$_pA$ is the subgroup of $A$ consisting of all torsion elements whose order is a power of $p$. One useful observation is that there exists a canonical isomorphism $_pH^*(-)\cong {_p[H^*(-)\lp]}$. Lastly, note that when we take tensor products of $\Z\lp$-modules, we do so over $\Z\lp$.
\end{notation}

In the following discussion, we outline our strategy for studying $H^*(BPU_n)$ for arbitrary $n$. Firstly, it is worth noting that the torsion-free component of $H^*(BPU_n)$ is already thoroughly understood. In addition, the $p$-primary subgroup of $H^*(BPU_n)$ is trivial for $p\nmid n$ from
the following results.

\begin{prop}[\cite{gu_zzz}, Proposition 1.1]
\label{pro:n-torsion}
    Suppose $x\in H^*(BPU_n)$ is a torsion class. Then there exists some $i\geq 0$ such that $n^ix = 0$.
\end{prop}

Therefore, to determine the graded abelian group structure of $H^s(BPU_n)$, it suffices to consider the $p$-primary subgroup $_pH^s(BPU_n)$ for $p\mid n$. 

In \cite{gu_zzz} and \cite{ZHANG2023108642}, we gived a comprehensive description of $H^s(BPU_n;\Z)\lp$ for $s<2p+5$ and $s=2p+6$ by proving that $_pH^s(BPU_n) = 0$ for $s = 2p+3,\ 2p+4$ and $2p+6$ (when $p$ is an odd prime). In addition, we listed the reason why the previous methods failed in determining $_pH^{2p+5} (BPU_n)$ at the end of \cite{ZHANG2023108642}.

In this paper, we improve our previous
methods and then extend the findings in \cite{gu_zzz} 
and \cite{ZHANG2023108642} to the range
$s<2p+9$
by computing $_pH^s(BPU_n)$ for $s = 2p+5,\ 2p+7$ and $ 2p+8$ for all $n$. Speak roughly, the difficulty at the end of \cite{ZHANG2023108642} is solved by 
an use of a new basis(Lemma \ref{lem:c bar = c}) and a technical
construction(Lemma \ref{lem:three boundaries}).
Our main theorem is as follows. 

\begin{thm}\label{thm: p>2, main thm, 2p+5,7,8 }
Let $p > 2$ be a prime number, and $n=p^rm$ for a positive integer $m$ co-prime to $p$. Then the $p$-primary subgroup of $H^{s}(BPU_n)$ in dimensions less than $2p+9$ is as follows:
\begin{enumerate}
    \item For $r > 0$, we have 
    \begin{equation*}
    _pH^s(BPU_n)\cong
    \begin{cases}
        \Z/p^r,\ s=3,\\
        \Z/p,\ s=2p+2 \text{ or } 2p+5, \\
        0,\  s<2p+9,\ s\neq 3, 2p+2, 2p+5.
    \end{cases}
    \end{equation*}
    \item For $r = 0$, we simply have $_pH^s(BPU_n) = 0$ for all $s\geq 0$. 
\end{enumerate}
\end{thm}

\begin{rem}
    When $p=2\mid n$, the results are
    quite different: $_pH^{2p+5}(BPU_n)={_p}H^{2p+8}(BPU_n)=\Z/p$. $_pH^{2p+7}(BPU_n)=\Z/p\oplus \Z/p$ or $\Z/p$,
    depending on whether $n\equiv 4
    (\opn{mod} 16)$.
\end{rem}

\subsection*{Organization of the paper}
In Section \ref{sec:spectral_sequence}, we introduce the Serre spectral sequence that we use to compute the cohomology of $BPU_n$. 
In Section \ref{sec: proof of lem}, we prove  two necessary lemmas. 
In Section \ref{sec: The proof of thm 1}, we 
complete the proof 
Theorem \ref{thm: p>2, main thm, 2p+5,7,8 }.

\section{The spectral sequences}\label{sec:spectral_sequence}

Our tool to compute the cohomology of $BPU_n$ is the Serre spectral sequence ${^U}\!E$ described in equation \eqref{eq: E2 of U convergence}. The same spectral sequence ${^U}\!E$ was also essential in the related computations presented in \cite{gu2019cohomology,gu_zzz,ZHANG2023108642}. This section serves to refresh the basic framework and computational outcomes for ${^U}\!E$.

\subsection{The Serre spectral sequence $^U\! E$}

The short exact sequence of Lie groups
$$1 \to S^1 \to U_n \to PU_n \to 1$$
induces a fiber sequence of their classifying spaces
$$BS^1 \to BU_n \to BPU_n$$
Notice that $BS^1$ has the homotopy type of the Eilenberg-Mac Lane space $K(\Z, 2)$, there is an associated fiber sequence
\begin{equation}\label{eq: fiber sequence U}
U: ~ BU_n \to BPU_n \to K(\Z, 3)
\end{equation}

We will use the Serre spectral sequence associated to \eqref{eq: fiber sequence U} to compute the cohomology of $BPU_n$.
For notational convenience, we denote this spectral sequence by $^U\!E$. The $E_2$ page of
$^U\!E$ has the form
\begin{equation}\label{eq: E2 of U convergence}
^U\! E^{s, t}_{2} = H^{s}(K(\mathbb{Z},3);H^{t}(BU_{n})) \Longrightarrow H^{s+t}(BPU_{n})
\end{equation}
To carry out actual computations with this spectral sequence, we need to know the cohomology of $K(\mathbb{Z},3)$ and $BU_{n}$.  Since the purpose of this paper is to study the $p$-primary subgroup of $H^*(BPU_n)$ for a fixed prime $p$, it suffices to know the $p$-local cohomology of $K(\Z, 3)$.

We summarize the $p$-local cohomology of $K(\Z, 3)$ in low dimensions as follows.  The original reference is \cite{cartan19551955}, also see \cite{tamanoi1999subalgebras} for a nice treatment.

\begin{prop}\label{prop: p local cohomology of KZ3 below 2p+5}
Let $p > 2$ be a prime. In degrees up to $2p+8$, we have
\begin{equation}\label{equation: p local cohomology of KZ3 below 2p+5}
H^{s}(K(\mathbb{Z},3))_{(p)} =
\begin{cases}
\Zp, & s = 0,\ 3,\\
\Z/p, &  s = 2p+2,\ 2p+5,\\
0, & s \leq 2p+8,  s \neq 0,\ 3,\ 2p+2,\ 2p+5.
\end{cases}
\end{equation}
where $x_1,~ y_{p, 0},~ x_1 y_{p, 0}$ are generators on degree $3, 2p+2, 2p+5$ respectively.
\end{prop}

\begin{rem}
Here the notations for the generators are taken from \cite[Proposition 2.14]{gu2019cohomology}.
\end{rem}

Also recall
\begin{equation}\label{equation: cohomology of BUn}
    H^{*}(BU_{n}) = \mathbb{Z}[c_1,c_2,\dots,c_n],\ |c_i|=2i
\end{equation}
In particular, $H^{*}(BU_{n})$ is torsion-free.  We have
\begin{equation}\label{equation: tensor form E2 of U}
    ^U\! E^{s, t}_{2} \cong H^{s}(K(\mathbb{Z},3)) \otimes H^{t}(BU_{n})
\end{equation}

\subsection{The auxiliary spectral sequences $^T\! E$ and $^K\! E$}

Instead of computing the differentials in $^U\!E$ directly, which is difficult in practice, our strategy is to compare $^U\!E$ with two auxiliary spectral sequences, which has simpler differential behaviors.  We now introduce the two auxiliary fiber sequences and their associated Serre spectral sequences.

Let $T^n$ be the maximal torus of $U_n$ with the inclusion denoted by
\[\psi: T^n\to U_n.\]
Passing to quotients over $S^1$, we have another inclusion of maximal torus
\[\psi': PT^n\to PU_n.\]
The quotient map $T^n\to PT^n$ fits into an exact sequence of Lie groups
\[1\to S^1\to T^n\to PT^n\to 1,\]
which induces another fiber sequence of their classifying spaces
\begin{equation}\label{eq: fiber sequence T}
T: ~ BT^n \to BPT^n \to K(\mathbb{Z}, 3)
\end{equation}
$T$ is our first auxiliary fiber sequence.

We also consider the path fibration for $K(\Z,3)$
\begin{equation}\label{eq: fiber sequence K}
K: ~ K(\mathbb{Z}, 2) \simeq BS^1 \to * \to K(\mathbb{Z}, 3)
\end{equation}
where $*$ denotes a contractible space. $K$ is our second auxiliary fiber sequence.

These fiber sequences fit into the following homotopy commutative diagram:
\begin{equation}\label{eq:3_by_3_diag}
    \begin{tikzcd}
        K\arrow[d,"\Phi"]:& BS^1\arrow[r]\arrow[d,"B\varphi"]& *\arrow[r]\arrow[d]&
        K(\Z,3)\arrow[d,"="]\\
        T:\arrow[d,"\Psi"]& BT^n\arrow[r]\arrow[d,"B\psi"]& BPT^n\arrow[r]\arrow[d,"B\psi'"]& K(\Z,3)\arrow[d,"="]\\
        U:& BU_n\arrow[r]& BPU_n\arrow[r]& K(\Z,3)
    \end{tikzcd}
\end{equation}
Here, the map $B\varphi: BS^1\to BT^n$ is induced by the diagonal map $\varphi: S^1\to T^n$.

We denote the Serre spectral sequences associated to $U$, $T$, and $K$ as $^U\! E$, $^T\! E$ and $^K\! E$ respectively.  We denote their corresponding differentials by ${^U\!d}_*^{*,*}$, ${^T\!d}_*^{*,*}$, and ${^K\!d}_*^{*,*}$ respectively.  When the actual meaning is clear from the context, we also simply denote the differentials by $d_*^{*,*}$.

In this paper, we compute differentials in $^U\!E$ by comparing them with the differentials in $^T\!E$ and $^K\!E$.  This is possible because: (1) we have explicit formulas for the maps between spectral sequences, and (2) we have a good understanding of the corresponding differentials in $^T\!E$ and $^K\!E$.

We first describe the comparison maps between $^U\!E$, $^T\!E$ and $^K\!E$.

Notice that we have
\begin{equation}\label{equation: cohomology of BTn}
    H^{*}(BT^{n}) = \mathbb{Z}[v_1,v_2,\dots,v_n],\ |v_i|=2.
\end{equation}

The induced homomorphism between cohomology rings is as follows:
\[B\varphi^*:H^*(BT^n) = \Z[v_1,v_2,\cdots,v_n] \to H^*(BS^1) = \Z[v],\ v_i\mapsto v.\]
The map $B\psi: BT^n\to BU_n$ induces the injective ring homomorphism
\begin{equation}\label{equation: cohomology map BUn to BTn}
    \begin{split}
        B\psi^*: H^*(BU_n) = \Z[c_1,\cdots,c_n] &\to H^*(BT^n) = \Z[v_1,\cdots,v_n],\\
        c_i &\mapsto \sigma_i(v_1,\cdots,v_n),
    \end{split}
\end{equation}
where $\sigma_i(t_1,t_2,\cdots,t_n)$ is the $i$th elementary symmetric polynomial in variables $t_1,t_2,\cdots,t_n$:
\begin{equation}\label{eq:sigma_def}
    \begin{split}
        & \sigma_0(t_1,t_2,\cdots,t_n) = 1,\\
        & \sigma_1(t_1,t_2,\cdots,t_n) = t_1+t_2+\cdots+t_n,\\
        & \sigma_2(t_1,t_2,\cdots,t_n) = \sum_{i<j}t_it_j,\\
        & \vdots\\
        & \sigma_n(t_1,t_2,\cdots,t_n) = t_1t_2\cdots t_n.
    \end{split}
\end{equation}

We also recall some important propositions regarding the higher differentials in $^K\!E$ and $^T\!E$.  The following result of differentials in $^K\!E$ is the starting point for relevant computations in $^T\!E$ and $^U\!E$.

\begin{prop}[\cite{gu2019cohomology}, Corollary 2.16]\label{prop:diff0}
 The higher differentials of ${^K\!E}_{*}^{*,*}$ satisfy
 \begin{equation*}
 \begin{split}
   &d_{3}(v)=x_1,\\
   &d_{2p^{k+1}-1}(p^k x_{1}v^{lp^{e}-1})=v^{lp^{e}-1-
   (p^{k+1}-1)}y_{p,k},\quad e > 0,\ \operatorname{gcd}(l,p)=1,\\
   &d_{r}(x_1)=d_{r}(y_{p,k})=0,\quad \textrm{for all }r,k>0
 \end{split}
\end{equation*}
and the Leibniz rule. 
\end{prop}
\begin{rem}
    Note there is a typo in the original reference, where the condition $k \geq e$ should be replaced by $e > k$.
\end{rem}

\begin{prop}[\cite{gu2019cohomology}, Proposition 3.2]\label{pro:diff1}
    The differential  $^{T}\!d_{r}^{*,*}$, is partially determined as follows:
\begin{equation}
 ^{T}\!d_{r}^{*,2t}(v_{i}^{t}\xi)={(B\pi_i)^{*}}({^K\!d}_{r}^{*,2t}(v^{t}\xi)),
\end{equation}
where $\xi\in {^{T}\!E}_{r}^{*,0}$, a quotient group of $H^*(K(\mathbb{Z}, 3))$, and $\pi_i: T^{n}\rightarrow S^1$ is the projection of the $i$th diagonal entry. In plain words, $^{T}\!d_{r}^{*,2t}(v_{i}^{t}\xi)$ is simply $^{K}\!d_{r}^{*,2t}(v^{t}\xi)$ with $v$ replaced by $v_i$.
\end{prop}
\begin{rem}
    Here we correct another typo in the original Proposition 3.2 in \cite{gu2019cohomology}, in which ``$~\xi\in {^{T}\!E}_{r}^{0,*}$~'' should be replaced by ``~$\xi\in {^{T}\!E}_{r}^{*,0}$~''.
\end{rem}

By comparing with the differentials in $^K\!E$, one could obtain the following results on differentials in $^T\!E$.

\begin{prop}[\cite{gu_zzz}, Lemma 3.1]
\label{prop: dT vn x1}
In the spectral sequence $^T\!E$, we have

\begin{enumerate}
    \item 
$v_{n}^{k}x_1 \in \opn{Im} {^T\!d}_3$
for $0 \le k \le p-2$ or $k=p$, 
\item 
$^{T}\!d^{3,*}_{2p-1} (v_{n}^{p-1} x_1) = y_{p,0}$.
\end{enumerate}
\end{prop}

\begin{prop}[\cite{gu2019cohomology}, Proposition 3.3]
\label{prop: d3 of T}
    \begin{enumerate}
        \item The differential $^{T}\!d_{3}^{0,t}$ is given by the ``formal divergence''
        \[\nabla=\sum_{i=1}^{n}(\partial/\partial v_i): H^{t}(BT^{n};R)\rightarrow H^{t-2}(BT^{n};R),\]
        in such a way that $^{T}\!d_{3}^{0,*}=\nabla(-)\cdot x_{1}.$ For any ground ring $R=\mathbb{Z}$ or $\mathbb{Z}/m$ for any integer $m$.
        \item The spectral sequence degenerates at ${{^T}\!E}^{0,*}_{4}$. Indeed, we have $^{T}\!E_{\infty}^{0,*}=$ $^{T}\!E_{4}^{0,*}=\operatorname{Ker}{^T\!d}_{3}^{0,*}=\mathbb{Z}[v_{1}-v_{n},\cdots, v_{n-1}-v_{n}]$.
    \end{enumerate}
\end{prop}

\begin{cor}[\cite{gu2019cohomology}, Corollary 3.4]\label{cor:d3}
Set $c_0 =1$. For $k\ge 1$, we have
    \[^{U}\! d_{3}^{0,*}(c_{k})=\nabla(c_{k})x_1=(n-k+1)c_{k-1}x_1.\]
\end{cor}

Recall that (\ref{equation: tensor form E2 of U}) implies
\begin{equation}
    \label{eq:p local tensor form E2 of U}
    (^U\! E^{s, t}_{2})_{(p)} \cong H^{s}(K(\mathbb{Z},3))_{(p)} \otimes H^{t}(BU_{n})
    \cong H^{s}(K(\mathbb{Z},3)) \otimes H^{t}(BU_{n})_{(p)}.
\end{equation}

\begin{notation}
    For the remainder of this paper, 
    we will use $^U\!E$, $^T\!E$ and $^K\!E$ to denote the
    corresponding $p$-localized Serre spectral sequence.
\end{notation}

\section{Proof of lemmas}
\label{sec: proof of lem}

In this section we prove two necessary lemmas.
Before the statement of lemmas,
we define a total order.

For an positive even $2t$, (\ref{eq:p local tensor form E2 of U}) shows the $\Zp$-module $^U \!E_3 ^{3,2t}$ is freely 
generated by elements of the form $cx_1$ for
\begin{equation}
    \label{equ:def of S't}
    c\in S'_t := \{ c_{i_1} c_{i_2}\cdots c_{i_l}|1\le i_1 \le
i_2\le\cdots\le i_l\le t,\ \Sigma_{k=1}^l i_k =t\}.
\end{equation}
We define a total order $\mathfrak{O}_t$ on these monomials 
as follows. We assert
$$c_{i_1} c_{i_2}\cdots c_{i_l}< c_{i'_1} c_{i'_2}\cdots c_{i'_s}$$
if and only if one of the following two conditions holds:

(1) $l<s$.

(2) $l=s$ and for the smallest $k$ such that $i_k\neq i'_k$,
we have $i_k<i'_k$.

To compare $cx_1 ,c'x_1 \in S'_t x_1$, we assert $cx_1 <c'x_1$
if and only if $c<c'$.
In addition, for the monomial $c_{i_1} c_{i_2}\cdots c_{i_l}x_1$, we define 
\begin{equation}
    \label{equ:bar cx}
    \overline{c_{i_1} c_{i_2}\cdots c_{i_l}x_1}
    =
    \begin{cases}
        ^U\!d_3 (c_{i_1} \cdots c_{i_{l-1}} c_{i_l +1}),
        \textrm{ if } p\nmid i_l,\\
        c_{i_1} c_{i_2}\cdots c_{i_l}x_1 ,\textrm{ if } p\mid i_l.
    \end{cases}
\end{equation}
We have an observation that 
any monomial in
$\overline{c_{i_1} c_{i_2}\cdots c_{i_l}x_1}$
is lower than or equal to $c_{i_1} c_{i_2}\cdots c_{i_l}x_1$.

For any $\mathfrak{c} \in S'_t $, let $\mathfrak{M}^t_{\mathfrak{c}}$ 
(or $\overline{\mathfrak{M}^t_{\mathfrak{c}}}$) be the $\Zp$-module freely 
generated by the elements $cx_1$ (or $\overline{cx_1}$) with $c\in S'_t$, $c\le \mathfrak{c}$. 
Then we have

\begin{lem}
    \label{lem:c bar = c}
    $\mathfrak{M}^t_{\mathfrak{c}}
    =\overline{\mathfrak{M}^t_{\mathfrak{c}}}$.
    In particular,
    as a $\Zp$-module, $^U \!E_3 ^{3,2t}$ can be freely
    generated by elements of the form $\overline{cx_1}$
    for $c\in S'_t$.
\end{lem}

\begin{proof}
The observation below \eqref{equ:bar cx}
implies $\overline{\mathfrak{M}^t_{\mathfrak{c}}}\subset \mathfrak{M}^t_{\mathfrak{c}}$. So it suffices to prove $\mathfrak{M}^t_{\mathfrak{c}} \subset \overline{\mathfrak{M}^t_{\mathfrak{c}}}$ for all $\mathfrak{c} \in S'_t$. 

We prove this claim by induction on the order. For the
lowest element $\mathfrak{c}=c_t$, the claim is clear
if $p\mid t$. If $p\nmid t$, $\overline{c_t x_1}
={^U\!d}_3 (c_{t+1}) = (n-t) c_t x_1$
and thus $c_t x_1 \in \overline{\mathfrak{M}^t_{c_t }}$.

Now assume we have proven the claim for all $c<\mathfrak{c}$. Note that $\mathfrak{M}^t_{\mathfrak{c}}
=\mathfrak{M}^t_{\mathfrak{c}'}\oplus \Zp\{\mathfrak{c}x_1\}$,
where $\mathfrak{c}'$ is the largest monomial less that
$\mathfrak{c}$. By assumption, $\mathfrak{M}^t_{\mathfrak{c}'}\subset \overline{\mathfrak{M}^t_{\mathfrak{c'}}}\subset \overline{\mathfrak{M}^t_{\mathfrak{c}}}$.
As before, we write $\mathfrak{c}=c_{i_1} c_{i_2}\cdots c_{i_l}$.

If $p\mid i_l$, $\mathfrak{c}x_1=\overline{\mathfrak{c}x_1}$. Thus $\Zp\{\mathfrak{c}x_1\}\subset \overline{\mathfrak{M}^t_{\mathfrak{c}}}$.

If $p\nmid i_l$, by (\ref{equ:bar cx}) and Corollary \ref{cor:d3}, we have
$$\overline{\mathfrak{c}x_1}=\ ^U\!d_3 (c_{i_1} \cdots c_{i_{l-1}} c_{i_l +1})=(n-i_l )\mathfrak{c}x_1 +
\textrm{(lower order terms)}.$$
By assumption and $p\nmid i_l$, $\mathfrak{c}x_1 \in \overline{\mathfrak{M}^t_{\mathfrak{c}}}
+\mathfrak{M}^t_{\mathfrak{c'}}
=\overline{\mathfrak{M}^t_{\mathfrak{c}}}
+\overline{\mathfrak{M}^t_{\mathfrak{c'}}}$.
Thus $\Zp\{\mathfrak{c}x_1\}\subset \overline{\mathfrak{M}^t_{\mathfrak{c}}}$.

\end{proof}

Before the statement of next lemma, we define two series of numbers in 
$\Z_{(p)}$. Let
\begin{equation*}
    \begin{split}
        & A_k =\prod _{j=k}^{(p-1)/2} \frac{n-j}{n-p+j}, 
        \textrm{ for } 1\le k \le (p-1)/2.\\
        & B_k =\prod _{j=k}^{(p+1)/2} \frac{n-j}{n-p-2+j}, 
        \textrm{ for } 3\le k \le (p+1)/2.
    \end{split}
\end{equation*}
Then we claim that

\begin{lem}
    \label{lem:three boundaries}
    There exist $X_1 \in {^U \!E}_3 ^{0,2p+4}$ and $X_2 ,X_3 \in 
    {^U \!E}_3 ^{0,2p+6}$ 
    such that
\begin{equation*}
    \begin{split}
        ^U\! d_3 (X_1) = & (-1)^{\frac{p-1}{2}} (p+2)nA_1 
        c_p c_1 x_1,\\
        ^U\! d_3 (X_2) = & (-1)^{\frac{p+1}{2}} 2B_3 (n-2) c_p c_2 x_1
        ,\\
        ^U\! d_3 (X_3) = & (-1)^{\frac{p-1}{2}} pA_1 c_p c_1 ^2 x_1
        + (\textrm{lower order terms}).
    \end{split}
\end{equation*}

\end{lem}

\begin{proof}

It suffices to construct $X_1, X_2, X_3$ satisfy 
the conditions. Let
\begin{equation*}
        \begin{split}
            X_1 = & c^2 _{\frac{p+1}{2}} c_1 -\frac{n}{n-\frac{p+1}{2}}c_{\frac{p+3}{2}}c_{\frac{p+1}{2}} +2 \Sigma_{k=1}
            ^{\frac{p-1}{2}} (-1)^{\frac{p-1}{2}-k+1}A_k 
            c_{p-k+1}c_k c_1 \\
            & + \Sigma_{k=2}
            ^{\frac{p-1}{2}} (-1)^{\frac{p-1}{2}-k} (p+2-2k)\frac{nA_k}{n-p+k-1} 
            c_{p-k+2}c_k ,\\
            X_2 = & c_{\frac{p+3}{2}}^2 +\Sigma_{k=3}
            ^{\frac{p+1}{2}} (-1)^{\frac{p+1}{2}-k+1}2B_k 
            c_{p-k+3}c_k,\\
            X_3 = & \frac{1}{n-1} (c_{\frac{p+1}{2}}^2 c_2 +
            \Sigma_{k=2}
            ^{\frac{p-1}{2}} (-1)^{\frac{p-1}{2}-k+1}2A_k 
            c_{p-k+1}c_k c_2)- \frac{1}{n-\frac{p+1}{2}} c_{\frac{p+3}{2}} c_{\frac{p+1}{2}} c_1\\
            & + \Sigma_{k=2}
            ^{\frac{p-1}{2}} (-1)^{\frac{p-1}{2}-k}(p-2+2k)\frac{A_k}{n-p+k-1} 
            c_{p-k+2}c_k c_1 \\
            & +(-1)^{\frac{p-1}{2}} 
            \frac{2A_2 }{n-p+1} c_p c_2 c_1.
        \end{split}
    \end{equation*}
    
We only verify the first formula in Lemma \ref{lem:three boundaries}. The other two formulas can be verified similarly. Note the definition of $A_k$ implies

\begin{equation}
    \label{equ: Ak, Bk}
    (n-p+k)A_k =(n-k)A_{k+1} \text{ for } 1\le k< (p-1)/2.
\end{equation}
Immediate calculation shows

\begin{equation*}
    \begin{split}
        ^U\! d_3 (X_1) = & 2(n-\frac{p-1}{2}) c_{\frac{p+1}{2}} 
        c_{\frac{p-1}{2}} c_1 x_1 + nc^2 _{\frac{p+1}{2}} x_1
        -nc^2 _{\frac{p+1}{2}} x_1 
        -nA_{\frac{p-1}{2}} c_{\frac{p+3}{2}} c_{\frac{p-1}{2}} x_1\\
        & + 2\Sigma_{k=1}^{\frac{p-1}{2}} (-1)^{\frac{p-1}{2}-k+1}
        A_k (n-p+k)c_{p-k} c_k c_1 x_1\\
        & + 2\Sigma_{k=1}^{\frac{p-1}{2}} (-1)^{\frac{p-1}{2}-k+1}
        A_k ((n-k+1)c_{p-k+1} c_{k-1} c_1 + nc_{p-k+1} c_k )x_1\\
        & + \Sigma_{k=2}^{\frac{p-1}{2}} (-1)^{\frac{p-1}{2}-k} 
        (p+2-2k) nA_k c_{p-k+1}c_k x_1 \\
        & + \Sigma_{k=2}^{\frac{p-1}{2}} (-1)^{\frac{p-1}{2}-k} 
        (p+2-2k) \frac{n(n-k+1)A_k }{n-p+k-1}
         c_{p-k+2}c_{k-1}x_1.
    \end{split}
\end{equation*}

Then by \eqref{equ: Ak, Bk}, 
\begin{equation*}
    \begin{split}
        ^U\! d_3 (X_1) = & 2(n-\frac{p-1}{2}) c_{\frac{p+1}{2}} 
        c_{\frac{p-1}{2}} c_1 x_1 
        -nA_{\frac{p-1}{2}} c_{\frac{p+3}{2}} c_{\frac{p-1}{2}} x_1\\
        & + 2\Sigma_{k=1}^{\frac{p-1}{2}} (-1)^{\frac{p-1}{2}-k+1}
        A_k (n-p+k)c_{p-k} c_k c_1 x_1\\
        & + 2\Sigma_{k=1}^{\frac{p-3}{2}} (-1)^{\frac{p-1}{2}-k}
        A_{k+1} (n-k)c_{p-k} c_{k} c_1 x_1 + (-1)^{\frac{p-1}{2}}
        4nA_1 c_p c_1 x_1\\
    \end{split}
\end{equation*}
\begin{equation*}
    \begin{split}
        & + \Sigma_{k=2}^{\frac{p-1}{2}} (-1)^{\frac{p-1}{2}-k} 
        (p-2k) nA_k c_{p-k+1}c_k x_1 \\
        & + \Sigma_{k=1}^{\frac{p-3}{2}} (-1)^{\frac{p-1}{2}-k+1} 
        (p-2k) \frac{n(n-k)A_{k+1} }{n-p+k}
         c_{p-k+1}c_{k}x_1\\
         = & (-1)^{\frac{p-1}{2}} (p+2)nA_1 
        c_p c_1 x_1.
    \end{split}
\end{equation*}
\end{proof}

\section{The proof of Theorem \ref{thm: p>2, main thm, 2p+5,7,8 }}
\label{sec: The proof of thm 1}

The purpose of this section is to prove 
Theorem \ref{thm: p>2, main thm, 2p+5,7,8 }. 
We will show later it's a consequence of
the following two propositions.

\begin{prop}
    \label{prop:E inf 3,2p+2 3,2p+4}
    In the spectral sequence, we have
    $${^U\!E}^{3,2p+2}_{\infty}={^U\!E}^{3,2p+4}_{\infty}=0.$$
\end{prop}

\begin{proof}[Proof]

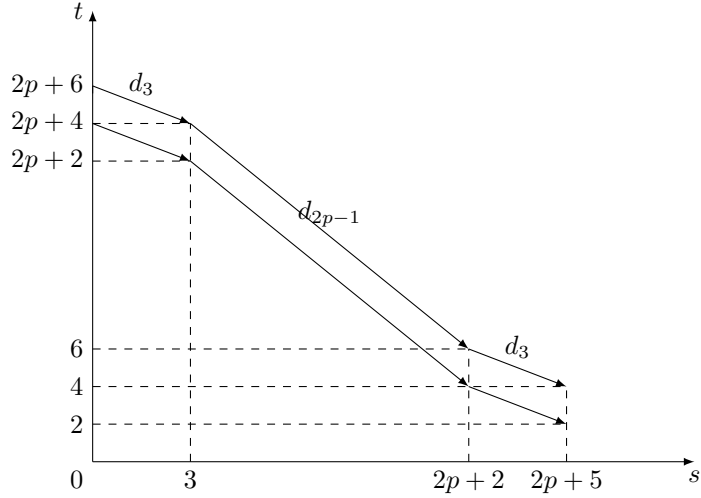
\begin{figure}
    \centering
    \begin{tikzpicture}
    \draw[-latex] (0,0) node [below left] {$0$} --(8,0) node [below] {$s$};
    \draw[-latex] (0,0)  --(0,6) node [left] {$t$};
    \draw[dashed](0,0.5) node [left] {$2$}--(6.3,0.5) ;
     \draw[dashed](0,1) node [left] {$4$}--(6.3,1) ;
      \draw[dashed](0,1.5) node [left] {$6$}--(5,1.5) ;
     \draw[dashed](0,4) node [left] {$2p+2$}--(1.3,4) ;
     \draw[dashed](0,4.5) node [left] {$2p+4$}--(1.3,4.5) ;
     \draw (0,5) node [left] {$2p+6$}--(0,4.9) ;
     \draw[dashed](1.3,0) node [below] {$3$}--(1.3,4.5) ;
     \draw[dashed](5,0) node [below] {$2p+2$}--(5,1.5) ;
     \draw[dashed](6.3,0) node [below] {$2p+5$}--(6.3,1) ;
     \draw[-latex] (0,5)--(0.65,4.75) node
     [above]{$d_3$}--(1.3,4.5);
     \draw[-latex] (0,4.5)  --(1.3,4);
     \draw[-latex] (1.3,4.5)-- (3.15,3) node
     [above]{$d_{2p-1}$} --(5,1.5);
     \draw[-latex] (1.3,4)--(5,1);
     \draw[-latex] (5,1.5)--(5.65,1.25)
     node [above]{$d_3$}--(6.3,1);
     \draw[-latex] (5,1)--(6.3,0.5);
    \end{tikzpicture}
    \caption{Some nontrivial differentials in the spectral sequence}
    \label{fig:differentials}
\end{figure}

Inspection of degrees shows that ${^U\!E}^{3,2p+2}_{*}$ can only 
support the differential $d_{2p-1}^{3,2p+2}$
and receive the differential
$d_{3}^{0,2p+4}$, as shown in Figure
\ref{fig:differentials}.
By Lemma \ref{lem:c bar = c}, 
$${^U\!E}^{3,2p+2}_3 =\Zp \{ \overline{cx_1}| \  c\in 
S'_{p+1} \} .$$
Recall the definition \eqref{equ:bar cx} of $\overline{cx_1}$, we have
$${^U\!E}^{3,2p+2}_4 = \frac{{^U\!E}^{3,2p+2}_3 }{\opn{Im} 
{^U\!d}_3^{0,2p+4}}
= \frac{\Zp \{ c_p c_1 x_1 \} }
{\opn{Im} {^U\!d}_3^{0,2p+4} \cap \Zp \{ c_p c_1 x_1 \} }.$$

If $p^2 \nmid n-p$, then ${^U\!d}_3^{0,2p+4}
(c_{p+1}c_1 -\frac{n}{n-p-1} c_{p+2})=(n-p) c_p 
c_1 x_1$. Thus $pc_p c_1 x_1 \in \opn{Im} {^U\!d}_3^{0,2p+4} \cap \Zp \{ c_p c_1 x_1 \}$.

If $p^2 \mid n-p$, then $p^2 \nmid n$. By 
Lemma \ref{lem:three boundaries} 
and $p\nmid A_1$, we have $pc_p c_1 x_1 \in \opn{Im} {^U\!d}_3^{0,2p+4} \cap \Zp \{ c_p c_1 x_1 \}$.

Now ${^U\!E}^{3,2p+2}_{2p-1}={^U\!E}^{3,2p+2}_4 $ is a
quotient of $\Zp \{ c_p c_1 x_1 \} / p\Zp \{ c_p c_1 x_1 \} \cong  \Z/p \{ $ $ c_p c_1 x_1 \}$. 
Recall that we have shown 
${^U\!d}^{3,2p+2}_{2p-1} (c_p c_1 x_1)= \binom{n-1}
{p-1} c_1^2 y_{p,0}\neq 0$
in the proof of Lemma 3.6 in \cite{ZHANG2023108642}. Hence we must have
$${^U\!E}^{3,2p+2}_{2p-1}= \Z/p \{ c_p c_1 x_1 \} \text{ and }
 {^U\!E}^{3,2p+2}_{\infty}= {^U\!E}^{3,2p+2}_{2p}
=\opn{Ker} {^U\!d}^{3,2p+2}_{2p-1}=0.$$

Inspection of degrees shows that ${^U\!E}^{3,2p+4}_{*}$ can only 
support the differential $d_{2p-1}^{3,2p+4}$
and receive the differential
$d_{3}^{0,2p+6}$, as shown in Figure
\ref{fig:differentials}.
By Lemma \ref{lem:c bar = c}, 
$${^U\!E}^{3,2p+4}_3 =\Zp \{ \overline{cx_1}| \  c\in 
S'_{p+2} \} .$$
The definition of $\overline{cx_1}$ shows
$${^U\!E}^{3,2p+4}_4 = \frac{{^U\!E}^{3,2p+4}_3 }{\opn{Im} 
{^U\!d}_3^{0,2p+6}}
= \frac{\Zp \{ c_p c_2 x_1 , c_p c_1^2 x_1 \} }
{\opn{Im} {^U\!d}_3^{0,2p+4} \cap \Zp \{ 
c_p c_2 x_1 , c_p c_1^2 x_1 \} }.$$
By Lemma \ref{lem:three boundaries} and  
$p\nmid B_3$, $c_p c_2 x_1 \in \opn{Im} {^U\!d}_3^{0,2p+4}$. 

Note the highest element in ${^U\!E}^{3,2p+4}_3$ lower than
$c_p c_1 ^2 x_1$ is 
$c_{\frac{p+3}{2}} c_{\frac{p+1}{2}} x_1 :=C$. 
Now we claim that $\mathfrak{M}^{p+2}_C \subset \opn{Im} {^U\!d}_3^{0,2p+4}$. 

By Lemma \ref{lem:c bar = c},
$\mathfrak{M}^{p+2}_C =\overline{\mathfrak{M}^{p+2}_C}$
$=\Zp \{ \overline{cx_1}|\ c\le C, c\in 
S'_{p+2} \}$. For $c\le C$ with $c\neq c_p c_2$, $\overline{cx_1} \in \opn{Im} {^U\!d}_3^{0,2p+4}$ from the definition \eqref{equ:bar cx}. Combined with $
\overline{c_p c_2 x_1} = c_p c_2 x_1 \in \opn{Im} {^U\!d}_3^{0,2p+4}$, we have
$\mathfrak{M}^{p+2}_C =\overline{\mathfrak{M}^{p+2}_C} \subset \opn{Im} {^U\!d}_3^{0,2p+4}.$

Note $p\nmid A_1$, Lemma \ref{lem:three boundaries}
implies $pc_p c_1^2 x_1 \in \opn{Im} {^U\!d}_3^{0,2p+4}+\mathfrak{M}^{p+2}_C \subset
\opn{Im} {^U\!d}_3^{0,2p+4}$.
Hence ${^U\!E}^{3,2p+4}_{2p-1}={^U\!E}^{3,2p+4}_4 $ is a
quotient of $\Zp \{ c_p c_1^2 x_1 \} / 
p\Zp \{ c_p c_1^2 x_1 \} \cong \Z/p \{ $ $ c_p c_1^2 x_1 \}$. 

Now we claim ${^U\!E}^{3,2p+4}_{\infty}={^U\!E}^{3,2p+4}_{2p-1}=0$. It suffices to
show ${^U\!d}_{2p-1}^{3,2p+4}(c_p c_1^2 x_1)
\neq 0$.

Recall the morphism of fiber sequences $\Psi$ introduced in \eqref{eq:3_by_3_diag}, and the induced morphism $\Psi^*: {^U\! E}\to {^T\! E}$ of spectral sequences. 
For $1\le i \le n$, let $v'_i =v_i-v_{n}$. It follows from Proposition \ref{prop: d3 of T} that the $v'_i$'s are permanent cycles. Then 

\begin{equation*}
    \begin{split}
         & \Psi^*{^U\!d}_{2p-1} ^{3,2p+4}(c_p c_1^2 x_1)\\
             =\ & {^T\!d}_{2p-1}^{3,2p+4}\Psi^{*}(c_p c_1^2 x_1)\\
             =\ & {^T\!d}_{2p-1}^{3,2p+4}((\sum_{n\geq i_1 > \cdots> i_p \geq 1}
             v_{i_1}\cdots v_{i_p})
             (\Sigma_{k=1}^n v_k)^2 x_1)\\
             =\ & {^T\!d}_{2p-1}^{3,2p+4}
             ((\sum_{n\geq i_1 > \cdots> i_p \geq 1}(v'_{i_1}+v_n)\cdots (v'_{i_p}+v_n))
             (\Sigma_{k=1}^n v'_k +nv_n)^2 x_1)\\
    \end{split}
\end{equation*}

\begin{equation*}
        \begin{split}
             =\ & {^T\!d}_{2p-1}^{3,2p+4}
             (\sum_{n\geq i_1 > \cdots> i_p \geq 1}\sum_{j=0}^{p}\sigma_j(v'_{i_1},\cdots,v'_{i_p})v_n^{p-j}
             (\Sigma_{k=1}^n v'_k +nv_n)^2 x_1).\\
        \end{split}
\end{equation*}
By Proposition \ref{prop: dT vn x1}, we have
\begin{equation*}
    \begin{split}
         & \Psi^*{^U\!d}_{2p-1}^{3,2p+4}(c_p c_1^2 x_1)\\
        =\ & {^T\!d}_{2p-1}^{3,2p+4}
        (\sum_{n\geq i_1 > \cdots> i_p \geq 1} (\sigma_1 (v'_{i_1},\cdots,v'_{i_p}) (\Sigma_{k=1}^n v'_k)^2 + 2n \sigma_2 (v'_{i_1},\cdots,v'_{i_p}) (\Sigma_{k=1}^n v'_k) \\
        & + n^2 \sigma_3 (v'_{i_1},\cdots,v'_{i_p}))
        v_n^{p-1}x_1)\\
        =\ & \sum_{n\geq i_1 > \cdots> i_p \geq 1} (\sigma_1 (v'_{i_1},\cdots,v'_{i_p}) (\Sigma_{k=1}^n v'_k)^2 + 2n \sigma_2 (v'_{i_1},\cdots,v'_{i_p}) (\Sigma_{k=1}^n v'_k) \\
        & + n^2 \sigma_3 (v'_{i_1},\cdots,v'_{i_p})
        y_{p,0}.\\
    \end{split}
\end{equation*}
Since $y_{p,0}$ is $p$-torison, we have
\begin{equation*}
    \begin{split}
        \Psi^*{^U\!d}_{2p-1}^{3,2p+4}(c_p c_1^2 x_1)
        =\ & \sum_{n\geq i_1 > \cdots> i_p \geq 1} \sigma_1 (v'_{i_1},\cdots,v'_{i_p}) (\Sigma_{k=1}^n v'_k)^2 
        y_{p,0}.\\
        =\ & \binom{n-1}{p-1} (\Sigma_{k=1}^n v'_k)^3 y_{p,0}\\
        =\ & \binom{n-1}{p-1} (\Sigma_{k=1}^n v_k -nv_n)^3 y_{p,0}.\\
        =\ & \binom{n-1}{p-1} (\Sigma_{k=1}^n v_k)^3   y_{p,0}\\
        =\ & \Psi^* (\binom{n-1}{p-1} c_1^3 y_{p,0}).
    \end{split}
\end{equation*}
Recall the map  
$\Psi^{*}: {^U\!E}_2^{*,*}\to {^T\!E}_2^{*,*}$ is injective.
We also know ${^U\!E}_{2p-1}^{2p+2,6}$ is a subgroup of 
${^U\!E}_2^{2p+2,6}$. Similarly ${^T\!E}_{2p-1}^{2p+2,6}$ 
is a subgroup of ${^T\!E}_2^{2p+2,6}$. Then we have 
\begin{equation}
    \label{equ:d2p-1 cpc1^2 x1}
    {^U\!d}_{2p-1}^{3,2p+4}(c_p c_1^2 x_1)=\binom{n-1}
{p-1} c_1^3 y_{p,0}\neq 0.
\end{equation}
\end{proof}

\begin{prop}
    \label{prop:E inf 2p+5,0 2p+5,2 2p+2,6}
    ${^U\!E}^{2p+5,2}_{\infty} ={^U\!E}^{2p+2,6}_{\infty} =0,\ 
    {^U\!E}^{2p+5,0}_{\infty} =\Z/p \{ y_{p,0}x_1 \}$
    or 0.
\end{prop}

\begin{proof}[Proof]
Clearly, the map ${^U\!d}^{2p+2,4}_3
: {^U\!E}^{2p+2,4}_3 \to {^U\!E}^{2p+5,2}_3$ is surjective, where
$${^U\!E}^{2p+2,4}_3 =\Z/p \{ c_2 y_{p,0}, 
c_1^2 y_{p,0}\},\ 
{^U\!E}^{2p+5,2}_3 =\Z/p \{ c_1 y_{p,0}x_1 \}.$$
Thus ${^U\!E}^{2p+5,2}_{\infty} ={^U\!E}^{2p+5,2}_4
=0$.

Inspection of degrees shows that ${^U\!E}^{2p+2,6}_{*}$ can only 
support the differential $d_{3}^{2p+2,6}$
and receive the differential
$d_{2p-1}^{3,2p+4}$. Immediate calculation
of ${^U\!d}^{2p+2,6}_3
: {^U\!E}^{2p+2,6}_3 \to {^U\!E}^{2p+5,4}_3$
shows ${^U\!E}^{2p+2,6}_{2p-1} = {^U\!E}^{2p+2,6}_4 = \opn{Ker} {^U\!d}^{2p+2,6}_3 =
\Z/p \{ c_1^3 y_{p,0} \}$,
where
$${^U\!E}^{2p+2,6}_3 =\Z/p \{ c_3 y_{p,0}, 
c_2c_1 y_{p,0}, c_1^3 y_{p,0} \},\ 
{^U\!E}^{2p+5,4}_3 =\Z/p \{ c_2 y_{p,0}x_1 ,
c_1^2 y_{p,0}x_1\}.$$
Now by \eqref{equ:d2p-1 cpc1^2 x1}, ${^U\!E}^{2p+2,6}_{\infty} ={^U\!E}^{2p+2,6}_{2p} =0$.

Inspection of degrees shows that ${^U\!E}^{2p+5,0}_{*}$ can't  
support any differential. Since 
${^U\!E}_2^{2p+5,0}=\Z/p \{y_{p,0}x_1\}$, ${^U\!E}^{2p+5,0}_{\infty}$
must be a quotient of ${^U\!E}_2^{2p+5,0}$, i.e. 
$\Z/p \{ y_{p,0}x_1 \}$ or $0$.

\end{proof}

\begin{proof}[Proof of Theorem \ref{thm: p>2, main thm, 2p+5,7,8 }]

The nontrivial entries in ${^U\!E}_2^{*,*}$ of
total degree $2p+5$ are ${^U\!E}_2^{3,2p+2}$ 
and ${^U\!E}_2^{2p+5,0}$.
By Proposition \ref{prop:E inf 3,2p+2 3,2p+4},
${^U\!E}_{\infty}^{3,2p+2}=0$. Hence, 
by Proposition \ref{prop:E inf 2p+5,0 2p+5,2 2p+2,6},
we have
$${_p H}^{2p+5}(BPU_n)= H^{2p+5}(BPU_n)_{(p)}={^U\!E}_{\infty}^{2p+5,0}=\Z/p 
\{ y_{p,0}x_1 \}
\text{ or } 0.$$
Then it suffices to show that ${_p H}^{2p+5}(BPU_n)\neq 0$. 
In the case $n=p$, by Theorem 3.6
of \cite{Vistoli+2007+181+227}, ${_p H}^{2p+5}(BPU_p)=\Z/p$. Thus, $y_{p,0}x_1\in {_p H}^{2p+5}(BPU_p)$ is nonzero.

As Lemma 7.2 of \cite{gu2019cohomology}, we consider 
the map $\Delta: BU_p \to BU_n$ given by the inclusion
$U_p \hookrightarrow U_n$:
\begin{equation*}
A\to 
    \begin{pmatrix}
         A & \cdots & 0 & 0\\
         0 & A & \cdots & \vdots\\
         \vdots & \cdots & \ddots & 0\\
         0 & \cdots & \cdots & A\\
    \end{pmatrix}.
\end{equation*}
This inclusion also yields the maps $PU_p \hookrightarrow PU_n$ and $\Delta': BPU_p \to BPU_n$. Now
as the proof in \cite{gu2019cohomology}, we can induce the following commutative diagram:
$$
\xymatrix{
& BU_p \ar@{->}[d]^{\Delta} \ar@{->}[r] & BPU_p
\ar@{->}[r] \ar@{->}[d]^{\Delta'} & K(\Z,3) 
\ar@{->}[d]^{=}\\
& BU_n \ar@{->}[r] & BPU_n
\ar@{->}[r]  & K(\Z,3)\\}
$$
This diagram induce a homomorphism of Serre spectral sequence whose restriction on the
bottom row of the $E_2$ pages is the identity.
Then $(\Delta')^* (y_{p,0}x_1) =y_{p,0}x_1$.
Note that the first $y_{p,0}x_1$ is in ${_p H}^{2p+5}(BPU_n)$ and the second is in
${_p H}^{2p+5}(BPU_p)$. Thus
$y_{p,0}x_1\in {_p H}^{2p+5}(BPU_n)$ is nonzero, i.e. ${_p H}^{2p+5}(BPU_n)\neq 0$.

The nontrivial entries in ${^U\!E}_2^{*,*}$ of
total degree $2p+7$ are ${^U\!E}_2^{3,2p+4}$ 
and ${^U\!E}_2^{2p+5,2}$.
By Proposition \ref{prop:E inf 3,2p+2 3,2p+4}
and Proposition \ref{prop:E inf 2p+5,0 2p+5,2 2p+2,6},
${^U\!E}_{\infty}^{3,2p+4}={^U\!E}_{\infty}^{2p+5,2}=0$. Hence, 
we have
$${_p H}^{2p+7}(BPU_n)= H^{2p+7}(BPU_n)_{(p)}=0.$$

The nontrivial entries in ${^U\!E}_2^{*,*}$ of
total degree $2p+8$ are ${^U\!E}_2^{0,2p+8}$ and ${^U\!E}_2^{2p+2,6}$.
By Proposition \ref{prop:E inf 2p+5,0 2p+5,2 2p+2,6},
${^U\!E}_{\infty}^{2p+2,6}=0$. Hence, 
$H^{2p+8}(BPU_n)_{(p)}= 
{^U\!E}_{\infty}^{0,2p+8}$ is a free $\Zp$-module, 
from which we deduce
$${_p H}^{2p+8}(BPU_n)=0.$$

\end{proof}

\subsection*{Acknowledgments}
The authors would like to thank Xing Gu
and Yu Zhang for the helpful discussions which motivated the current project. 
The authors were supported by the National Natural Science Foundation of China (No. 12001474; 12261091).  All authors contribute equally.

\bibliographystyle{plain}
\bibliography{ref}

\end{document}